\documentclass[10pt]{amsart}
\input{epsf}
\usepackage{amssymb,latexsym,longtable}
\topskip  \numberwithin{equation}{section}
\newtheorem{theorem}{Theorem}[section]

\newtheorem{remark}[theorem]{Remark}

\newtheorem{example}[theorem]{Example}

\newcommand{\da}{\mbox{-}}

\newcommand{\simref}[1]{\stackrel{#1}{\sim}}

\begin{document}

\pagenumbering{arabic}
\pagestyle{headings}
\def\sof{\hfill\rule{2mm}{2mm}}
\def\llim{\lim_{n\rightarrow\infty}}

\title[On avoidance of patterns by words]{On avoidance of patterns of the form $\sigma\da\tau$ by words over a finite alphabet}

\maketitle

\begin{center}

\author{Toufik Mansour\\
\small Department of Mathematics, University of Haifa, 31905 Haifa, Israel\\[-0.8ex]
\small\texttt{tmansour@univ.haifa.ac.il}\\[1.8ex]
Mark Shattuck\\
\small Department of Mathematics, University of Tennessee, Knoxville, TN 37996\\[-0.8ex]
\small\texttt{shattuck@math.utk.edu}\\[1.8ex]}

\end{center}

\begin{abstract}

Vincular or dashed patterns resemble classical patterns except that some of the letters within an occurrence are required to be adjacent.  We prove several infinite families of Wilf-equivalences for $k$-ary words involving vincular patterns containing a single dash, which explain the majority of the equivalences witnessed for such patterns of length four.  When combined with previous results, numerical evidence, and some arguments in specific cases, we obtain the complete Wilf-classification for all vincular patterns of length four containing a single dash. In some cases, our proof shows further that the equivalence holds for multiset permutations since it is seen to respect the number of occurrences of each letter within a word. Some related enumerative results are provided for patterns $\tau$ of length four, among them generating function formulas for the number of members of $[k]^n$ avoiding any $\tau$ of the form $11a\da b$.
\end{abstract}

\noindent{\em Keywords:} pattern avoidance, vincular patterns, $k$-ary words

\noindent 2010 {\em Mathematics Subject Classification:} 05A15, 05A19

\def\P{POGP}
\def\A{\mathcal{A}}
\def\SS{\frak S}
\def\Ps{POGPs}
\def\mn{\mbox{-}}
\def\newop#1{\expandafter\def\csname #1\endcsname{\mathop{\rm #1}\nolimits}}
\newop{MND}
\section{Introduction}

The Wilf-classification of patterns is a general question in enumerative combinatorics that has been addressed on several discrete structures mostly in the classical case.  We refer the reader to such texts as \cite{HM,Ki1,M} and the references contained therein.  Vincular patterns (also called ``generalized'' or ``dashed'' patterns) resemble classical patterns, except that some of the letters must be consecutive within an occurrence.  The Wilf-classification of vincular patterns of length three for permutations was completed by Claesson \cite{C}.  Progress on the classification of length four vincular patterns for permutations was made in \cite{Ki2,E,B,Ka}, and now all but possibly two equivalences have been shown (see \cite{BS}).

The analogous question concerning avoidance by $k$-ary words has also been addressed.  Burstein and Mansour \cite{BM,BM2} considered the Wilf-classification of vincular patterns of length three for $k$-ary words and found generating function formulas for the number of members of a class in several cases. The comparable problem for length four patterns has been partially addressed in \cite{Ka}, where several equivalences were shown to follow from a more general result on partially commutative monoids generated by a poset.  In this paper, we continue work started in \cite{Ka} and complete the Wilf-classification of vincular patterns of type $(3,1)$ or $(2,2)$ for $k$-ary words.

Let $[k]=\{1,2,\ldots,k\}$.  By a $k$-ary word, we mean a member of $[k]^n$.  The reduction of a $k$-ary word $\alpha$ having $\ell$ distinct letters is the word in $[\ell]$ obtained by replacing all copies of the $i$-th smallest letter of $\alpha$ with $i$ for each $i \in [\ell]$ and is denoted by $\text{red}(\alpha)$.  For example, $\text{red}(694614)=342312$.  The words $\alpha$ and $\beta$ are said to be \emph{isomorphic}, denoted $\alpha \equiv \beta$, if $\text{red}(\alpha)=\text{red}(\beta)$.

By a \emph{pattern} $\sigma=\sigma_1\sigma_2\cdots \sigma_m$, we will mean a member of $[\ell]^m$ for some $\ell$ and $m$ in which each letter in $[\ell]$ appears at least once.  A word $w=w_1w_2\cdots w_n \in [k]^n$ \emph{contains} $\sigma$ as a \emph{classical pattern} if there is a subsequence $w_{i_1}w_{i_2}\cdots w_{i_m}$ for $1 \leq i_1<i_2<\cdots <i_m\leq n$ such that $w_{i_1}w_{i_2}\cdots w_{i_m}\equiv \sigma$.  \emph{Vincular} patterns are similar to classical patterns except that some of the indices $i_j$ are required to be consecutive.  One may regard a vincular pattern as a pair $(\sigma,X)$ for a word $\sigma$ and a set of adjacencies $X \subseteq [m-1]$.  Then we say that $w$ \emph{contains} an \emph{occurrence} (or \emph{copy}) of $(\sigma,X)$ if there exists a subsequence $w_{i_1}w_{i_2}\cdots w_{i_m}$ of $w$ that is isomorphic to $\sigma$ with $i_{j+1}-i_j=1$ for each $j \in X$.  Otherwise, we say that $w$ \emph{avoids} $(\sigma,X)$.

One often expresses $(\sigma,X)$ as the permutation $\sigma$ with a dash between $\sigma_j$ and $\sigma_{j+1}$ if $j \notin X$ and refer to ``the vincular pattern $\sigma$'' without explicitly writing $X$.  For example, $(1342,\{2\})$ is written $1\da34\da2$.  The word $24356213$ contains an occurrence of $1\da34\da2$ as witnessed by the subsequence $2563$, but the subsequence $2453$ is not an occurrence of $1\da34\da2$ since the $4$ and $5$ are not adjacent.  \emph{Classical} patterns are those of the form $(\sigma,\varnothing)$ where no adjacencies are required, while \emph{consecutive} or \emph{subword} patterns are those of the form $(\sigma,[m-1])$ where copies of $\sigma$ must appear as subfactors $w_iw_{i+1}\cdots w_{i+m-1}\equiv \sigma$. Vincular patterns of the form $\sigma_1\da\sigma_2\da\cdots\da\sigma_r$, where each $\sigma_i$ has length $s_i$ for $1 \leq i \leq r$, are said to be of \emph{type} $(s_1,s_2,\ldots,s_r)$.

The reverse of a $k$-ary word $w=w_1w_2\cdots w_n$ is given by $w^r=w_nw_{n-1}\cdots w_1$, and the complement by $w^c=(k+1-w_1)(k+1-w_2)\cdots (k+1-w_n)$.  The reverse of a vincular pattern $\sigma$ having $\ell$ distinct letters is obtained by reading the letters and dashes together in reverse order, while the complement is obtained by replacing each copy of the letter $j$ with $\ell+1-j$ for all $j$, maintaining the relative positions of the dashes.  For example, we have $(13\da23\da4)^r=4\da32\da31$ and $(13\da23\da4)^c=42\da32\da1$.  Observe that a word $w$ contains a pattern $\sigma$ if and only if $w^r$ contains $\sigma^r$ and likewise for the complement.

We will make use of the following notation.  Given a generalized pattern $\sigma$, let $\mathcal{A}_\sigma(n,k)$ denote the subset of $[k]^n$ whose members avoid $\sigma$ and let $a_\sigma(n,k)=|\mathcal{A}_\sigma(n,k)|$.  Two patterns $\sigma$ and $\tau$ are \emph{Wilf-equivalent} if $a_\sigma(n,k)=a_\tau(n,k)$ for all $n$ and $k$, and we denote this by $\sigma\sim\tau$.  From the preceding remarks on symmetry it is clear that
$\sigma \sim \sigma^r \sim \sigma^c \sim \sigma^{rc}$.  The \emph{Wilf-equivalence class} of a pattern $\sigma$ consists of all patterns $\tau$ such that $\tau\sim\sigma$.

It will be useful to refine the sets $\mathcal{A}_\sigma(n,k)$ according to various prefixes.  Given $w=w_1w_2\cdots w_m$, let $\mathcal{A}_\sigma(n,k;w)$ denote the subset of $\mathcal{A}_\sigma(n,k)$ whose members $\pi=\pi_1\pi_2\cdots\pi_n$ satisfy $\pi_1\pi_2\cdots\pi_m=w_1w_2\cdots w_m$ and let $a_\sigma(n,k;w)=|\mathcal{A}_\sigma(n,k;w)|$.

The paper is divided as follows.  Section 2 proves several infinite families of Wilf-equivalences for vincular patterns containing a single dash.  These results have as corollaries the majority of the non-trivial equivalences witnessed for such patterns of length four.  One of the results provides a partial answer to a question raised in \cite{Ka} of finding analogues for $k$-ary words and compositions of some known equivalences on permutations.  In some cases, the family of equivalences is seen to respect the number of occurrences of each letter within a word and hence is actually an equivalence for \emph{multiset permutations} (see \cite{HM2}).

Section 3 proves additional equivalences for patterns of length four and completes the Wilf-classification for patterns of length four.  All non-trivial equivalences, up to symmetry, involving patterns of type $(3,1)$ or $(2,2)$ are listed.  Section 4 provides some related enumerative results.  Among them are formulas for the generating functions for the sequences $a_\tau(n,k)$ where $k$ is fixed and $\tau$ is any pattern of the form $11a\da b$.  We make use of primarily combinatorial methods to prove our results in the second and third sections, while in the last our methods are more algebraic.  In that section, we adapt the \emph{scanning-elements algorithm} described in \cite{FM}, a technique which has proven successful in enumerating length three pattern avoidance classes for permutations, to the comparable problem involving type $(3,1)$ vincular patterns and $k$-ary words.

\section{Some General Equivalences}

In this section, we show that some general families of equivalences hold concerning the avoidance of a vincular pattern having one dash.  We will make use of the following notation and terminology. A sequence of consecutive letters within $\pi=\pi_1\pi_2\cdots \pi_n \in [k]^n$ starting with the letter in the $i$-th position and isomorphic to a subword $\sigma$ will be referred to as an \emph{occurrence of $\sigma$ at index $i$}.  An \emph{ascent} (resp., \emph{descent}) refers to an index $i$ such that $\pi_i<\pi_{i+1}$ (resp., $\pi_i>\pi_{i+1}$).  By a \emph{non-ascent}, we will mean an index $i$ such that $\pi_i\geq \pi_{i+1}$.  A sequence of consecutive letters within $\pi$ possessing a common property will be often be referred to as a \emph{string} (of letters) having the given property.  If $i \geq 1$, then $a^i$ stands for the sequence consisting of $i$ copies of the letter $a$.  Finally, if $m$ and $n$ are positive integers, then $[m,n]=\{m,m+1,\ldots,n\}$ if $m \leq n$ and $[m,n]=\varnothing$ if $m>n$.

Our first result provides a way of extending the equivalence of a pair of subword patterns to a pair of longer vincular patterns containing the subwords.  By a \emph{monotonic} subword, we will mean one whose letters are either weakly increasing or decreasing.

\begin{theorem}\label{th8}
Let $\tau\sim \rho$ be subwords each having largest letter $s$. Let $\sigma$ be a non-empty monotonic subword pattern whose largest letter is $t$. Then $\tau\da (\sigma+s) \sim \rho\da (\sigma+s)$ and $\sigma\da(\tau+t) \sim \sigma\da(\rho+t)$.
\end{theorem}
\begin{proof}
Note that the second equivalence follows from the first by reverse complementation. Let $\tau'=\tau\da (\sigma+s)$ and $\rho'=\rho\da (\sigma+s)$.  We will show the first equivalence by constructing a bijection between $\mathcal{A}_{\tau'}(n,k)$ and $\mathcal{A}_{\rho'}(n,k)$.  Clearly, we may assume $s+t\leq k$. By a \emph{$b$-occurrence of $\sigma$} in $\pi$, we will mean one in which the smallest letter in the occurrence is $b$.  Throughout the rest of the proof, let $c=|\sigma|$ denote the length of the pattern $\sigma$

Let $g$ be a bijection which realizes the equivalence between $\tau$ and $\rho$.  Let $\pi=\pi_1\pi_2\cdots \pi_n \in \mathcal{A}_{\tau'}(n,k)$.  If $\pi$ contains no $b$-occurrences of $\sigma$ for which $b \in [s+1,k]$, then let $f(\pi)=\pi$.  Otherwise, we define sequences $\{a_j\}_{j\geq1}$ and $\{\ell_j\}_{j\geq 1}$ as follows.  Let $a_1$ denote the largest $i\in[s+1,k]$ for which there is an $i$-occurrence of $\sigma$ in $\pi$ and suppose the rightmost $a_1$-occurrence of $\sigma$ occurs at index $\ell_1$.  Define $a_j$ and $\ell_j$ inductively for $j>1$ as follows: Let $a_j$ denote the largest letter $i \in [s+1,a_{j-1}-1]$ for which there exists an $i$-occurrence of $\sigma$ whose first letter occurs to the right of position $\ell_{j-1}+c-1$ in $\pi$, with the rightmost $a_j$-occurrence of $\sigma$ occurring at index $\ell_{j}$.  Note that $a_1>a_2>\cdots> a_r$ and $\ell_1<\ell_2<\cdots <\ell_r$ for some $r \geq 1$.

If $\sigma$ is (weakly) increasing and $1 \leq j \leq r$, then no section of $S=\pi_{\ell_j}\pi_{\ell_j+1}\cdots\pi_{\ell_j+c-1}$ can form part of an occurrence of $\rho'$ coming prior to the dash, for otherwise there would be an $i$-occurrence of $\sigma$ for which $i>a_j$ occurring to the right of $S$ in $\pi$, which is impossible.  The same conclusion concerning sections of $S$ holds if $\sigma$ is (weakly) decreasing.  For if a section of $S$ comprised some of the letters coming prior to the dash in an occurrence of $\rho'$, then there would be at least one letter to the left of the dash that is larger than a letter to the right of it, which is impossible.

We then apply $g$ separately to any and all non-empty strings of letters in the alphabet $[a_{j+1}-1]$ occurring amongst  $\pi_{\ell_{j}+c}\pi_{\ell_{j}+c+1}\cdots\pi_{\ell_{j+1}-1}$ for each $1 \leq j<r$ as well as to any strings of letters in the alphabet $[a_1-1]$ occurring amongst $\pi_{1}\pi_{2}\cdots\pi_{\ell_1-1}$.  We leave any letters occurring to the right of the rightmost $a_r$-occurrence of $\sigma$ in $\pi$ unchanged.  The resulting word $f(\pi)$ avoids $\rho'$ and the mapping $f$ is seen to be a bijection.  Note that $\pi$ avoids $\tau'$ if and only if each of the strings described above avoids $\tau$ and likewise for $\rho'$.  To define $f^{-1}$, one then applies $g^{-1}$ to the same strings of letters to which one applied $g$.

For example, if $\tau=12$, $\rho=21$, and $\sigma=123$, then $\tau'=12\da345$ and $\rho'=21\da345$.  If $n=32$, $k=8$, and
$$\pi=\overline{431}7\underline{678}\overline{32}456\overline{332}5\underline{457}\overline{21}345\overline{21}\underline{358}434 \in \mathcal{A}_{\tau'}(32,8),$$
then $a_1=6$, $a_2=4$, and $a_3=3$, with $\ell_1=5$, $\ell_2=17$, and $\ell_3=27$. Note that the mapping $g$ is given by the reversal in this example.  Changing the order of the overlined strings would then imply
$$f(\pi)=\overline{134}7\underline{678}\overline{23}456\overline{233}5\underline{457}\overline{12}345\overline{12}\underline{358}434 \in \mathcal{A}_{\rho'}(32,8).$$
Here the rightmost $a_i$-occurrences of $\sigma$ in $\pi$ and $f(\pi)$ are underlined and the strings of letters that they govern within these words are overlined.
\end{proof}

The equivalence of the patterns $\rho$ and $\sigma$ over $[k]^n$ is said to be \emph{strong} if it respects the number of occurrences of each letter, that is, if $\rho$ and $\sigma$ are equivalent when viewed as patterns over permutations of the same \emph{multiset}. Observe that the proof of the preceding theorem shows further that strong equivalence of $\tau$ and $\rho$ is preserved by $\tau'$ and $\rho'$.

The following avoidance result is known for permutations (see \cite{E, Ki2}): if $\alpha \sim \beta$ are subword patterns of length $k$, then (i) $\alpha \da (k+1) \sim \beta \da (k+1)$, (ii) $\alpha \da (k+2)(k+1) \sim \beta \da (k+2)(k+1)$, and (iii) $\alpha\da (k+1)(k+2) \sim \alpha \da (k+2)(k+1)$. In \cite[Section 4.2]{Ka}, Kasraoui raised the question of finding analogues of (i)-(iii) for $k$-ary words and compositions.  Theorem \ref{th8} provides the requested equivalences in the case of words upon taking $\sigma=1$ or $\sigma=21$ in the first statement and taking $\tau=12$, $\rho=21$ in the second.  Note that for (iii), we obtain a proof of the result only in the case when $\alpha$ is monotonic.  However, modifying the proof of Theorem \ref{th9} below will imply (iii) for all $\alpha$.

Our next result concerns an infinite family of equivalences involving vincular patterns in which the sequence of letters prior to the dash contain exactly one peak and no valleys.

\begin{theorem}\label{th10}
Let $\ell \geq 3$ and $\tau=\tau_1\tau_2\cdots \tau_\ell\da\tau_{\ell+1}$ denote a vincular pattern of length $\ell+1$ such that $\tau_1\leq\tau_2\leq \cdots \leq \tau_{j-1}<\tau_j>\tau_{j+1}\geq \tau_{j+2}\geq \cdots \geq \tau_{\ell}$, where $2 \leq j \leq \ell-1$ and $\tau_j<\tau_{\ell+1}$.  Let $\rho=\rho_1\rho_2\cdots \rho_\ell\da \rho_{\ell+1}$ denote the pattern obtained from $\tau$ by interchanging the letters $\tau_j$ and $\tau_{\ell+1}$.  Then $\tau\sim \rho$ and this equivalence respects the first letter statistic.
\end{theorem}
\begin{proof}
We will show that $a_\tau(n,k;a)=a_\rho(n,k;a)$ for each $a \in [k]$ by induction on $n$ and $k$.  We may restrict attention to the subsets of $\mathcal{A}_\tau(n,k;a)$ and $\mathcal{A}_\rho(n,k;a)$ whose members contain each letter of $[k]$ at least once, since the complementary subsets have the same cardinality by induction on $k$.  Given $1 \leq i \leq \ell$, let $w_1w_2\cdots w_{i}$ be a $k$-ary word of length $i$ with $w_1=a$ such that $w_1w_2\cdots w_{i-1}\equiv \tau_1\tau_2\cdots \tau_{i-1}$ but $w_1w_2\cdots w_{i}\not\equiv \tau_1\tau_2\cdots \tau_{i}$.  Given a word $\alpha$, let $a_\tau^*(n,k;\alpha)$ denote the cardinality of the subset $\mathcal{A}_\tau^*(n,k;\alpha)\subseteq\mathcal{A}_\tau(n,k;\alpha)$ whose members contain every letter in $[k]$ at least once and likewise for $\rho$. By induction on $n$, we have that $a_\tau^*(n,k;w_1w_2\cdots w_i)=a_\rho^*(n,k;w_1w_2\cdots w_i)$, upon deleting the first $i-1$ letters and considering how many distinct letters occurring amongst the first $i-1$ positions occur again beyond the $(i-1)$-st position.

Let us assume now concerning the pattern $\tau$ that at least one of the inequalities $\tau_1\leq \tau_2$ or $\tau_{\ell-1}\geq \tau_\ell$ holds strictly.  We will first complete the proof in this case.  Suppose $w=w_1w_2\cdots w_\ell$ is a $k$-ary word isomorphic to $\tau_1\tau_2\cdots \tau_\ell$ with $w_1=a$.  Note that in order for the set $\mathcal{A}_\tau^*(n,k;w_1w_2\cdots w_\ell)$ to be non-empty, we must have $w_j=k$, and in order for the set $\mathcal{A}_\rho^*(n,k;w_1w_2\cdots w_\ell)$ to be non-empty, we must have $w_j=m$, where $m=\max\{w_{j-1},w_{j+1}\}+1$.  If $w_j=k$ in $w$, then let $\widetilde{w}$ be the word obtained from $w$ by replacing $w_j=k$ with $w_j=m$ and leaving all other letters of $w$ unchanged.

We now partition $\mathcal{A}_\tau(n,k;w)$ and $\mathcal{A}_\rho(n,k;\widetilde{w})$ as follows.  Let $S$ denote a subset of distinct letters amongst $w_1w_2\cdots w_\ell$ excluding the letter $w_\ell$, where $|S|=r$ and $|S|\cap[w_\ell-1]=s$ for some $r$ and $s$, and let $\mathcal{A}_\tau^*(n,k;w,S)\subseteq\mathcal{A}_\tau^*(n,k;w)$ consist of those words in which the only letters amongst those in the prefix $w$ that fail to occur beyond the $(\ell-1)$-st position are those in $S$.  Let $\widetilde{S}$ be a subset of $[k]$ obtained from $S$ by either replacing $k$ with $m$ if $k \in S$ or doing nothing if $k \not\in S$, and let $\mathcal{A}_\rho^*(n,k;\widetilde{w},\widetilde{S})\subseteq\mathcal{A}_\rho^*(n,k;\widetilde{w})$
consist of those words in which the only letters amongst those in the prefix $\widetilde{w}$ that fail to occur beyond the $(\ell-1)$-st position are those in $\widetilde{S}$.

Note that $a_\tau^*(n,k;w)=\sum_S|\mathcal{A}_\tau^*(n,k;w,S)|$ and $a_\rho^*(n,k;\widetilde{w})=\sum_S|\mathcal{A}_\rho^*(n,k;\widetilde{w},\widetilde{S})|$, where the sums range over all possible $S$ with $r$ and $s$ varying.  By deletion of the first $\ell-1$ letters and induction, we have
$$a_\tau^*(n,k;w,S)=a_\tau^*(n-\ell+1,k-r;w_\ell-s)=a_\rho^*(n-\ell+1,k-r;w_\ell-s)=a_\rho^*(n,k;\widetilde{w},\widetilde{S})$$
for each possible choice of $S$, which implies $a_\tau^*(n,k;w)=a_\rho^*(n,k;\widetilde{w})$.  Note that as $w$ ranges over all possible prefixes isomorphic to $\tau_1\tau_2\cdots\tau_\ell$ within members of $\mathcal{A}_\tau^*(n,k;a)$, we have that $\widetilde{w}$ ranges over the comparable prefixes within members of $\mathcal{A}_\rho^*(n,k;a)$.

Collecting all of the cases above in which a prefix starts with the letter $a$ shows that $a_\tau^*(n,k;a)=a_\rho^*(n,k;a)$ and completes the induction.  This then completes the proof in the case when  at least one of the inequalities $\tau_1\leq \tau_2$ or $\tau_{\ell-1}\geq \tau_\ell$ holds strictly.

Suppose now that neither of these inequalities holds strictly for $\tau$ and furthermore that $\tau_1=\tau_2=\cdots=\tau_b$ and $\tau_{\ell-b+1}=\tau_{\ell-b+2}=\cdots=\tau_\ell$, where $1<b<\min\{j,\ell-j+1\}$ is maximal.  In this case, we proceed by induction and show the following:\\

\noindent (i) $a_\tau^*(n,k;a)=a_\rho^*(n,k;a)$, and \\

\noindent (ii) $a_\tau^*(n,k;a^b)=a_\rho^*(n,k;a^b)$.\\

By the preceding, we only need to show that (ii) follows by induction in the case when the words in question have prefix $w=w_1w_2\cdots w_\ell$ isomorphic to $\tau_1\tau_2\cdots \tau_\ell$ with $w_1=a$.  Using the same notation as above, we have by the induction hypothesis for (ii) that
$$a_\tau^*(n,k;w,S)=a_\tau^*(n-\ell+b,k-r;(w_\ell-s)^b)=a_\rho^*(n-\ell+b,k-r;(w_\ell-s)^b)=a_\rho^*(n,k;\widetilde{w},\widetilde{S})$$
for all $S$, which implies $a_\tau^*(n,k;w)=a_\rho^*(n,k;\widetilde{w})$.  Letting $w$ vary implies (ii) in the $n$ case and completes the induction of (i) and (ii), which finishes the proof.
\end{proof}

Our next result shows the equivalence of a family of patterns each differing from one another by a single letter.

\begin{theorem}\label{th6}
If $i \geq 2$ and $c,d \in [i]$, then $12\cdots i\da c \sim 12\cdots i \da d$.
\end{theorem}
\begin{proof}
We'll show that $\mathcal{A}_{12\cdots i\da c}(n,k;a)$ and $\mathcal{A}_{12\cdots i\da d}(n,k;a)$ have the same cardinality for all $a \in [k]$ by induction.  Clearly, both sets contain the same number of strictly increasing members, so let us restrict attention to those that aren't.  Let $\alpha=a_1a_2\cdots a_j a_{j+1}$, where $j \geq 1$ and $a=a_1<a_2<\cdots<a_j\geq a_{j+1}$.  If $j<i$, then both sets contain the same number of members starting with the prefix $\alpha$, by induction, upon deleting the first $j$ letters, so let us assume $j \geq i$.  Then members of $\mathcal{A}_{12\cdots i\da c}(n,k;\alpha)$ can only contain letters in $[k]-\{a_c,a_{c+1},\ldots,a_{c+j-i}\}$ beyond the $j$-th position and a similar remark applies to $\mathcal{A}_{12\cdots i\da d}(n,k;\alpha)$.

Let $s$ be given and suppose $\pi \in \mathcal{A}_{12\cdots i \da c}(n,k;\alpha)$ has $s$ as its $(j+1)$-st letter.  Let $\widetilde{s}$ denote the value corresponding to $s$ when the set $[a_j]-\{a_c,a_{c+1},\ldots,a_{c+j-i}\}$ is standardized.  Pick $t \in [a_j]-\{a_d,a_{d+1},\ldots,a_{d+j-i}\}$ such that the standardization $t^*$ of $t$ relative to this set satisfies $t^*=\widetilde{s}$.  Let $\alpha'=a_1a_2\cdots a_j$.  By induction, we have
$$a_{12\cdots i\da c}(n,k;\alpha's)=a_{12\cdots i\da c}(n-j,k-j+i-1;\widetilde{s})=a_{12\cdots i\da d}(n-j,k-j+i-1;t^*)=a_{12\cdots i\da d}(n,k;\alpha't).$$
Note that as $s$ varies over all the possible values in $[a_j]-\{a_c,a_{c+1},\ldots,a_{c+j-i}\}$, $t$ varies over all possible values in $[a_j]-\{a_d,a_{d+1},\ldots,a_{d+j-i}\}$.  It follows that the subset of $\mathcal{A}_{12\cdots i\da c}(n,k;a)$ whose members have their first non-ascent at position $j$ has the same cardinality as the comparable subset of $\mathcal{A}_{12\cdots i\da d}(n,k;a)$.  Allowing $j$ to vary implies $a_{12\cdots i\da c}(n,k;a)=a_{12\cdots i\da d}(n,k;a)$ as desired.
\end{proof}

Our next result concerns the equivalence of a family of patterns of the same length consisting of distinct letters that are strictly increasing except for the final letter to the right of the dash.

\begin{theorem}\label{th11}
If $r \geq 3$ and $2 \leq u,v \leq r$, then $1\cdots(u-1)(u+1)\cdots (r+1)\da u \sim 1\cdots(v-1)(v+1)\cdots(r+1)\da v$.
\end{theorem}
\begin{proof}
It is enough to show the equivalence in the case when $v=u+1$.  Let $\tau=1\cdots (s-1)(s+1)\cdots (r+1)\da s$ and $\rho=1\cdots s(s+2)\cdots (r+1)\da (s+1)$, where $2 \leq s \leq r-1$.  We show by induction on $n$ that $a_\tau(n,k;a)=a_\rho(n,k;a)$ for all $k$ and any $a \in [k]$.  Deletion of the first $i$ letters implies by induction that the sets $A_\tau(n,k;a)$ and $A_\rho(n,k;a)$ have the same cardinality if $w=w_1w_2\cdots w_{i+1}$ is a $k$-ary word satisfying $a=w_1<w_2<\cdots<w_i\geq w_{i+1}$ where $1 \leq i \leq r-1$.  So assume $\pi=\pi_1\pi_2\cdots \pi_n \in \mathcal{A}_\tau(n,k;a)$ starts with at least $r-1$ ascents and contains at least one non-ascent.  That is, for some $t\geq r-1$ and $x_1<x_2<\cdots<x_t$, we have $\alpha:=\pi_1\pi_2\cdots \pi_{t+2}=a(a+x_1)(a+x_2)\cdots(a+x_t)b$, where $b \leq a+x_t$.  Deleting the first $t+1$ letters and noting some forbidden letters in $[k]$, we have $$a_\tau(n,k;\alpha)=a_\tau(n-t-1,k-x_{s+t-r}+x_{s-2}+t-r+2;\widetilde{b})
 ,$$
where $x_0=0$ and $\widetilde{b}$ denotes the relative size of $b$ in the set $[a+x_t]-\cup_{i=0}^{t-r+1}[a+x_{i+s-2}+1,a+x_{i+s-1}-1]$.

Let $d=x_{s+t-r+1}-x_{s+t-r}$. Define the word $\alpha'$ by $\alpha'=a(a+y_1)(a+y_2)\cdots (a+y_t)c$, where
$$y_i=\begin{cases} x_i, & \text{if} \text{~}\text{~} 1\leq i \leq s-2 \text{ or } s+t-r+1\leq i \leq t; \\ x_{i-1}+d, & \text{if} \text{~}\text{~} s-1 \leq i \leq s+t-r,\\ \end{cases}$$
the $x_i$ are as before, and $c \leq a+y_t$ is to be determined.  Reducing letters we have
$$a_\rho(n,k;\alpha')=a_\rho(n-t-1,k-x_{s+t-r}+x_{s-2}+t-r+2;c^*),$$
where $c^*$ denotes the relative size of $c$ in the set
$$[a+x_t]-\cup_{i=0}^{t-r+1}[a+x_{i+s-2}+d+1,a+x_{i+s-1}+d-1].$$
Observe that there are $a+x_t+x_{s-2}-x_{s+t-r}+t-r+2$ possibilities for $\widetilde{b}$ and the same number of possibilities for $c^*$.  So given $b$, if we pick $c$ such that $c^*=\widetilde{b}$, it follows by induction that $a_\tau(n,k;\alpha)=a_\rho(n,k;\alpha')$ for such a choice of $c$.  Note that as $\alpha$ varies over all prefixes of the given form within members of $\mathcal{A}_\tau(n,k;a)$ containing at least $r-1$ ascents prior to the first non-ascent, we have that $\alpha'$ varies over the comparable sequences within the members of $\mathcal{A}_\rho(n,k;a)$.  Since both sets clearly contain the same number of strictly increasing members, it follows that $a_\tau(n,k;a)=a_\rho(n,k;a)$, which completes the induction and proof.
\end{proof}

Our final result of this section provides a way of generating equivalences from subwords and explains several equivalences witnessed for patterns of small size.

\begin{theorem}\label{th9}
If $\sigma$ is a subword whose largest letter is $r$, then $\sigma\da r(r+1)\sim \sigma \da(r+1)r$.
\end{theorem}
\begin{proof}
Let $\tau=\sigma\da r(r+1)$ and $\rho=\sigma\da (r+1)r$. By an \emph{$s$-occurrence of} $\sigma$, $\tau$ or $\rho$ within $\pi$, we will mean an occurrence in which the role of the $r$ is played by the actual letter $s$ in $\pi$. We will define a bijection between $\mathcal{A}_\tau(n,k)$ and $\mathcal{A}_\rho(n,k)$.  Suppose $\pi=\pi_1\pi_2\cdots \pi_n \in \mathcal{A}_{\tau}(n,k)$.  If $\pi$ has no $s$-occurrences of $\sigma$ for any $s$, then let $f(\pi)=\pi$.  Otherwise, let $k_1$ be the smallest $s \in [k]$ such that $\pi$ contains at least one $s$-occurrence of $\sigma$.  Suppose that the leftmost $k_1$-occurrence of $\sigma$ within $\pi$ occurs at index $j_1$.  Let $S_1$ denote the subsequence of $\pi$ consisting of any letters in $[k_1,k]$ amongst $\pi_{j_1+a}\pi_{j_1+a+1}\cdots\pi_n$, where $a=|\sigma|$ denotes the length of $\sigma$.  Then $S_1$ may be empty or it may consist of one or more non-empty strings of letters in $[k_1,k]$ separated from one another by letters in $[k_1-1]$.  Within each of these strings, any $k_1$ letters must occur after any letters in $[k_1+1,k]$ in order to avoid an occurrence of $\tau$ in $\pi$.  We move any copies of $k_1$ from the back of each string to the front, and let $\pi_1$ denote the resulting word.

Note that $\pi_1$ contains no $k_1$-occurrences of $\rho$ and that $\pi_1=\pi$ if $\pi$ contains no such occurrences of $\rho$.  Furthermore, we have that $\pi_1$ contains no $s$-occurrences of $\rho$ for any $s \in [k_1-1]$ as well.  To see this, suppose to the contrary that it contains a $b$-occurrence of $\rho$ for some $b \in [k_1-1]$.  By the minimality of $k_1$, such an occurrence of $\rho$ must involve some letter $c\in[k_1+1,k]$ that was moved in the transition from $\pi$ to $\pi_1$ playing the role of the $r+1$.  But then $\pi$ would have contained a $b$-occurrence of $\rho$ (with $k_1$ playing the role of the $r+1$), contradicting the minimality of $k_1$.

We define the $\pi_i$ recursively as follows.  Define $k_i$, $i>1$, by letting $k_i$ be the smallest $s \in [k_{i-1}+1,k]$ for which there exists an $s$-occurrence of $\sigma$ in $\pi_{i-1}$.  Let $j_i$ denote the index of the leftmost $k_i$-occurrence of $\sigma$ in $\pi_{i-1}$ and $S_i$ be the subsequence of $\pi_{i-1}$ consisting of letters in $[k_i,k]$ lying to the right of the $(j_i+a-1)$-st position.  Again, we move any copies of $k_i$ from the back to the front within each non-empty string of letters in $S_i$.  Let $\pi_i$ denote the resulting word.  Note that no letters in $[k_i-1]$ are moved in this step (or in later ones).  It may be verified that $\pi_i$ contains no $s$-occurrences of $\rho$ for any $s \in [k_i]$ since no $s$-occurrences of $\sigma$ for $s<k_i$ are created in the transition from $\pi_{i-1}$ to $\pi_i$ as well as no $k_i$-occurrences of $\sigma$ to the left of the leftmost such occurrence in $\pi_{i-1}$.  Since the $k_i$ strictly increase, the process must terminate after a finite number of steps, say $t$.  Let $f(\pi)=\pi_t$.  Note that $f(\pi)$ avoids $s$-occurrences of $\rho$ for all $s \in [k_t]$, and hence avoids $\rho$, since $\pi_t$ doesn't even contain any $s$-occurrences of $\sigma$ for $s \in [k_t+1,k]$.

The mapping $f$ is seen to be a bijection.  To define its inverse, consider whether or not $f(\pi)$ contains an $s$-occurrence of $\sigma$ for any $s$, and if so, the largest $s \in [k]$ for which there is an $s$-occurrence of $\sigma$ in $f(\pi)$, along with the position of the leftmost such $s$-occurrence.  Note that this $s$ must be $k_t$, by definition, for if not and there were a larger one, then the procedure described above in generating $f(\pi)$ would not have terminated after $t$ steps.  Within strings of letters in $[k_t,k]$ occurring to the right of the leftmost $k_t$-occurrence of $\sigma$ in $f(\pi)$, re-order the letters so that the $k_t$'s succeed rather than precede any other letters.  Note that this undoes the transformation step from $\pi_{t-1}$ to $\pi_t$, where $\pi_0=\pi$.  On the word that results, repeat this procedure, considering the largest $s \in [k_t-1]$ occurring in some $s$-occurrence of $\sigma$ and then the leftmost position of such an occurrence.

For example, if $n=30$, $k=6$, $\sigma=11$, and
$$\pi=\pi_0=215562213422\underline{1}\underline{1}65354435436542\overline{1}\overline{1} \in \mathcal{A}_{11\da12}(30,6),$$
then
$$\pi_1=21556\underline{22}134\overline{22}11116535443543654\overline{2}, \quad \pi_2=215562212234111126535\underline{44}35\overline{4}365\overline{4},$$

$$\pi_3=21\underline{55}62212234111126\overline{5}3544345346\overline{5},$$
and
$$f(\pi)=\pi_4=215562212234111125635443453456 \in \mathcal{A}_{11\da21}(30,6).$$

(Letters corresponding the leftmost $k_i$-occurrence, $1 \leq i \leq 4$, of $\sigma$ are underlined in each step, and any $k_i$ that must be moved in the $i$-th step are overlined.)
\end{proof}

Note that the preceding proof shows further the strong equivalence of the patterns $\sigma\da r(r+1)$ and $\sigma \da(r+1)r$.

\section{Patterns of length four}

In this section, we consider some further equivalences for patterns of length four containing a single dash. Combining the results of this section with those from the prior and from \cite{Ka} will complete the Wilf-classification for avoidance of a single $(3,1)$ or $(2,2)$ pattern by $k$-ary words.  The Wilf-equivalence tables for these patterns are given at the end of this section.

\subsection{Further results}

In this subsection, we provide further equivalences concerning the avoidance of $(3,1)$ patterns by $k$-ary words.  Our first result generalizes the equivalence $113\da2\sim133\da2$.

\begin{theorem}\label{th3}
If $i\geq 2$, then $1^i3\da2 \sim 13^i\da2$.
\end{theorem}
\begin{proof}
For simplicity, we prove only the $i=2$ case, the general case following by making the suitable modifications.  By a $113$ or $133$ avoiding $k$-ary word, we will mean, respectively, one which avoids either the $112$ or $122$ subword except for possible adjacencies of the form $aa(a+1)$ or $a(a+1)(a+1)$.  We first show that the number of members of $\mathcal{A}_\tau(n,k)$ starting with $a$ and ending with $b$ where $a$ and $b$ are given is the same for $\tau=113$ as it is for $\tau=133$.  Suppose $\pi \in \mathcal{A}_{113}(n,k)$ starts with the letter $a$ and ends in $b$.  First assume that there are no descents in $\pi$.  Then $\pi$ may be expressed as $\pi=a_1^{n_1}a_2^{n_2}\cdots a_r^{n_r}$ for some $r\geq 1$ where $a=a_1<a_2<\cdots<a_r=b$, $n_i \geq 1$ for $i \in [r]$, and $n_i=1$ if $a_{i+1}>a_i+1$.  Similarly, if $\rho \in \mathcal{A}_{133}(n,k)$ starts with $a$ and ends in $b$ and contains no descents, then we may write $\rho=a_1^{m_1}a_2^{m_2}\cdots a_r^{m_r}$ for some $r$ where $m_{i+1}=1$ if $a_{i+1}>a_i+1$.

We define a bijection $f$ between the words $\pi$ and $\rho$ described above as follows. Let $1 \leq i_1<\cdots<i_{\ell-1}<i_\ell=r$ be the set of indices such that $a_{i_j+1}=a_{i_j}+1$ for $1 \leq j \leq \ell-1$.  Note that the exponent $n_i$ must be one if $i \in [r]-\{i_1,i_2,\ldots,i_\ell\}$, with the $n_{i_j}$ exponents unrestricted.  We define the $m_i$ exponents for indices lying in the interval $[i_j+1,i_{j+1}]$, $0 \leq j \leq \ell-1$, by setting $m_{i_j+1}=n_{i_{j+1}}$ and $m_t=1$ if $i_{j}+1<t\leq i_{j+1}$, where $i_0=0$.  This yields the desired bijection between members of $\mathcal{A}_{113}(n,k)$ and $\mathcal{A}_{133}(n,k)$ whose first and last letters are given and containing no descents.

Now suppose $\pi=\pi_1\pi_2\cdots\pi_n \in \mathcal{A}_{113}(n,k)$ contains at least one descent.  For this case, we will proceed by induction on $n$.  We may assume $n \geq 4$ since the result is clear if $1 \leq n \leq 3$.  Then $a=\pi_1\leq \pi_2\leq \cdots \leq\pi_{i'-1}>\pi_{i'}$ for some $i'>1$.  Likewise, suppose $\rho=\rho_1\rho_2\cdots \rho_n \in \mathcal{A}_{133}(n,k)$ starts with $a$ and contains at least one descent, with $i'-1$ the position of the leftmost descent.  By the bijection in the previous paragraph, we have that the number of possible subsequences $\pi_1\pi_2\cdots\pi_{i'-1}$ in which $\pi_{i'-1}$ is a given number $s$ is the same as the number of possible subsequences $\rho_1\rho_2\cdots\rho_{i'-1}$ in which $\rho_{i'-1}=s$.  If $\pi_{i'}=\rho_{i'}=t<s$ is given, then there are the same number of possibilities for $\pi_{i'}\pi_{i'+1}\cdots \pi_n$ as there are for $\rho_{i'}\rho_{i'+1}\cdots \rho_n$ by induction since both classes of sequences are to begin with $t$ and end with $b$, with the former avoiding $113$ and the latter $133$.  Letting the position of the first descent vary as well as the values of $s$ and $t$, it follows that the number of members of $\mathcal{A}_\tau(n,k)$ whose first and last letters are prescribed is the same for $\tau=113$ as it is for $\tau=133$.

We now show that $a_{113\da2}(n,k;a)=a_{133\da2}(n,k;a)$ for each $a \in [k]$ by induction.  Clearly the members of either set not containing an occurrence of $113$ or $133$, respectively, are equinumerous by the preceding.  So assume $\pi=\pi_1\cdots\pi_n \in \mathcal{A}_{113\da2}(n,k;a)$ contains at least one occurrence of the $113$ subword and that $\ell$ is the smallest index $i$ such that $\pi_{i-1}\pi_i\pi_{i+1}$ is an occurrence of $113$.  Let $\pi_\ell=u$ and $\pi_{\ell+1}=v$, where $v>u+1$.  Similarly, define $\ell$, $u$, and $v$ in conjunction with $\rho=\rho_1\cdots \rho_n \in \mathcal{A}_{133\da2}(n,k;a)$ and the subword $133$.  We will show that the respective subsets of $\mathcal{A}_{113\da2}(n,k;a)$ and $\mathcal{A}_{133\da2}(n,k;a)$ containing such $\pi$ and $\rho$ have the same cardinality.

To do so, first note that the number of $113$-avoiding words $\alpha$ of length $\ell-1$ starting with $a$ and ending with $u$ equals the number of comparable $133$-avoiding words $\beta$, by the preceding.  Furthermore, for any pair $\alpha$ and $\beta$ as described, we have
$$a_{113\da2}(n,k;\alpha u v)=a_{113\da2}(n-\ell,k-v+u+1;u+1)=a_{133\da2}(n-\ell,k-v+u+1;u+1)=a_{133\da2}(n,k;\beta vv),$$
by induction.  Allowing the prefixes $\alpha u v$ and $\beta v v$ to vary over all possible lengths and values of $u$ and $v$, it follows that $a_{113\da2}(n,k;a)=a_{133\da2}(n,k;a)$ for all $a$, as desired.
\end{proof}

\begin{remark}\label{rem31}
A similar proof may be given for the equivalences $1^i2\da1 \sim 12^i\da1$ for $i \geq 2$.
\end{remark}

We next show the equivalence of the patterns $132\da1$ and $132\da2$.

\begin{theorem}\label{th1}
We have $132\da1\sim132\da2$.
\end{theorem}
\begin{proof}
We will simultaneously show by induction on $n$ for all $k$ the following:\\

\noindent(i) $a_{132\da1}(n,k;a)=a_{132\da2}(n,k;a)$ for all $a$, and\\

\noindent (ii) $a_{132\da1}(n,k;xyz)=a_{132\da2}(n,k;xyz)$ for any $x<z<y$.\\

We may assume $n>4$, since both statements clearly hold if $n \leq 4$.  We first show (i).  To do so, first note that both $\mathcal{A}_{132\da1}(n,k;a)$ and $\mathcal{A}_{132\da2}(n,k;a)$ contain as a subset all of the $132$-avoiding members of $[k]^n$ starting with the letter $a$, so let us assume all words under consideration contain at least one occurrence of $132$.  Let $\pi=\pi_1\cdots\pi_n\in \mathcal{A}_{132\da1}(n,k;a)$ denote a $k$-ary word and $i$ be the smallest index such that $\pi_i\pi_{i+1}\pi_{i+2}$ is a $132$ subword.  If $i>1$, then let $\rho=\pi_1\cdots \pi_{i-1}$ and $u$, $v$, and $w$ be the letters comprising the first occurrence of $132$.  Then we have by induction that
$$a_{132\da1}(n,k;\rho uvw)=a_{132\da1}(n-i+1,k;uvw)=a_{132\da2}(n-i+1,k;uvw)=a_{132\da2}(n,k;\rho uvw)$$
in this case.  So assume $i=1$.  Note that this case in showing (i) is equivalent to showing (ii).

To do so, suppose $\pi=\pi_1\pi_2\cdots \pi_n$ is either a member of  $\mathcal{A}_{132\da1}(n,k;xyz)$ or $\mathcal{A}_{132\da2}(n,k;xyz)$, where $x<z<y$. Let $r$ denote the fourth letter of $\pi$.  We consider cases on $r$.  If $r\leq z$, then there are $z-1$ possibilities concerning $r$ for members of either set.  Thus if $1 \leq r < x$, we have
$$a_{132\da1}(n,k;xyzr)=a_{132\da1}(n-3,k-1;r)=a_{132\da2}(n-3,k-1;r)=a_{132\da2}(n,k;xyzr),$$
while if $x < r \leq z$, we have
$$a_{132\da1}(n,k;xyzr)=a_{132\da2}(n-3,k-1;r-1)=a_{132\da2}(n-3,k-1;r-1)=a_{132\da2}(n,k;xyz(r-1)).$$

If $r>z$, then let $j$ be the smallest index greater than three such that $\pi_{j+1} \leq \pi_j$.  Note that if no such index exists, then $z=\pi_3<\pi_4<\cdots < \pi_n$ and there are clearly an equal number of options concerning membership in either set.  Otherwise, we have $z=\pi_3<\pi_4<\cdots<\pi_j \geq \pi_{j+1}$ for some $j \geq 4$.  Let $\alpha=\pi_4\pi_5\cdots \pi_{j+1}$.  If $s=\pi_{j+1}=\pi_j$, then
$$a_{132\da1}(n,k;xyz\alpha)=a_{132\da1}(n-j,k-1;s-1)=a_{132\da2}(n-j,k-1;s-1)=a_{132\da2}(n,k;xyz\alpha).$$
Next suppose $s=\pi_{j+1}<\pi_j$ but that $\pi_{j-1}\pi_j\pi_{j+1}$ does not form an occurrence of $132$.  Then $a_{132\da1}(n,k;xyz\alpha)=a_{132\da2}(n,k;xyz\alpha)$ in this case, by the preceding arguments, upon considering the further subcases $1 \leq s <x$, $x<s \leq z$, or $z <s \leq \pi_{j-1}$.

Finally, suppose $\pi_{j+1}<\pi_j$ and that $\pi_{j-1}\pi_j\pi_{j+1}=uvw$ is an occurrence of $132$.  For this case, we consider an equivalent description of (ii) above.  Note that (ii) holds if and only if $a_{132\da1}(n-2,k-1;z-1)$ equals the number of $132\da2$ avoiding $k$-ary words of length $n-2$ starting with $z$ in which the only occurrence of $z$ is the first letter.  Deleting the first $j$ letters of $\pi \in \mathcal{A}_{132\da1}(n,k;xyz\alpha)$ implies in this case that $a_{132\da1}(n,k;xyz\alpha)=a_{132\da1}(n-j,k-2;w-2)$.  On the other hand, deleting the first $j$ letters of $\pi \in \mathcal{A}_{132\da2}(n,k;xyz\alpha)$ implies that $a_{132\da2}(n,k;xyz\alpha)$ equals the number of members of $[k-1]^{n-j}$ starting with $w-1$ in which the only occurrence of $w-1$ is the first letter.  By the reformulation of (ii) just described and induction, we have $a_{132\da1}(n,k;xyz\alpha)=a_{132\da2}(n,k;xyz\alpha)$ in this case as well.  Combining all of the cases above completes
  the inductive proof of (ii), as desired.
\end{proof}

\begin{remark}\label{rem11}
Comparable proofs to the one above may be given for the equivalences
$$213\da1\sim213\da2\mbox{ and }122\da1\sim 122\da2.$$
\end{remark}

The final result of this subsection features an outwardly dissimilar looking pair of equivalences.

\begin{theorem}\label{th4}
We have
$134\da2 \sim 143\da 2,~124\da3 \sim 214\da3, \mbox{ and } 142\da3 \sim 241\da3.$
\end{theorem}
\begin{proof}
By an \emph{$i$-occurrence of} $134\da2$ (or $143\da2$) within some word $w$, we will mean one in which the 4 position corresponds to the actual letter $i$ in $w$.  We first define a mapping $f$ between $\mathcal{A}_{134\da2}(n,k)$ and $\mathcal{A}_{143\da2}(n,k)$ as follows.  Let $\pi \in \mathcal{A}_{134\da2}(n,k)$.  First consider whether or not there are any $k$-occurrences of $143\da2$ within $\pi$.  If there are any, we first remove the leftmost such occurrence of $143\da2$ in $\pi$ by interchanging the letter $k$ within this occurrence with its successor.  Then remove the leftmost $k$-occurrence of $143\da2$ in the resulting word in the same manner and repeat until there are no $k$-occurrences left.  Let $f_1(\pi)$ denote the word so obtained.  Note that $f_1(\pi)$ contains no $k$-occurrences of $143\da2$ and that $f_1(\pi)=\pi$ if $\pi$ avoids such occurrences to start with.  Proceeding inductively, if $j>1$, we remove any $(k-j)$-occurrences of $143\da2$ from $f_j(\pi)$ by interchanging the letters corresponding to the $4$ within these occurrences with their successors, starting with the leftmost and working from left to right.  Note that $f_j(\pi)$ contains no $(k+1-i)$-occurrences of $143\da2$ for any $i \in [j]$.  The process is seen to terminate after $k-3$ steps.  Set $f(\pi)=f_{k-3}(\pi)$.  Note that $f(\pi)$ avoids $143\da2$.

For example, if $n=25$, $k=6$, and
$$\pi=\underline{365}6\underline{264}1\underline{1635}6143254\underline{163}423 \in \mathcal{A}_{134\da2}(25,6),$$
then
$$f_1(\pi)=3566246113566143\underline{254}136423,\quad f_2(\pi)=3566246113566\underline{143}245136423,$$
and
$$f(\pi)=f_3(\pi)=3566246113566134245136423 \in \mathcal{A}_{143\da2}(25,6).$$

We now show that $f$ is a bijection.  First note that each $k$ in $\pi \in \mathcal{A}_{134\da2}(n,k)$ corresponding to a $k$-occurrence of $143\da2$ is translated to the right (several places, if necessary) until it is no longer part of an occurrence, starting with the leftmost such $k$.  Furthermore, each $k$ that must be moved is separated from any other such $k$'s by at least one letter less than $k$.  Since each time a $k$ is moved to the right, the leftmost $k$-occurrence of $143\da2$ is also shifted to the right, the number of steps involved in the process of removing all $k$-occurrences of $143\da2$ from a word is bounded above by the total number of letters, and hence it terminates.  Similar remarks apply to any $i \in [4,k-1]$ that must be translated to the right in the transition from $f_{k-i}(\pi)$ to $f_{k-i+1}(\pi)$.  Finally, any $i$ that is moved in the $(k-i+1)$-st step remains undisturbed in possible later steps when only letters strictly less than $i$ may be moved.

So in order to define the inverse of $f$, we can start with the rightmost $4$-occurrence of $134\da2$ in $f(\pi)$, if there is one, and move the $4$ contained within it to the left by interchanging its position with its predecessor's until it is no longer part of a $4$-occurrence of the pattern.  Repeat for each subsequent $4$ that is part of a $4$-occurrence of $134\da2$, going from right to left.  Then repeat this process for each $i>4$, ending with $i=k$.  It may be verified that the mapping from $\mathcal{A}_{143\da2}(n,k)$ to $\mathcal{A}_{134\da2}(n,k)$ that results from performing this procedure is the inverse of $f$.

A similar bijection may be given for the equivalence $124\da3\sim214\da3$.  By a \emph{$j$-occurrence of} $124\da3$ (or $214\da3$), we mean one in which the $1$ position corresponds to the actual letter $j$.  Suppose $\pi \in \mathcal{A}_{124\da3}(n,k)$.  Consider any $1$-occurrences of $214\da 3$ in $\pi$ and switch them to $124\da3$ by interchanging the letters corresponding to the $1$ and $2$ positions, starting with the rightmost $1$-occurrence of $214\da3$ and working from right to left.  Subsequently replace any $j$-occurrences of $214\da3$ in $\pi$ with $124\da3$ for $j>1$, ending with $j=k-3$.  The resulting word is seen to belong to $\mathcal{A}_{214\da3}(n,k)$.  The mapping is reversed by first replacing any $(k-3)$-occurrences of $124\da3$ with $214\da3$, starting with the leftmost and working from left to right, and then repeating this for any smaller $j$-occurrences of $124\da3$, ending with $j=1$.

A bijection comparable to the first one above applies to the equivalence $142\da3 \sim 241 \da 3$.  By an $i$-occurrence of either pattern in a $k$-ary word, we mean one in which the $2$ corresponds to the letter $i$, where $2 \leq i \leq k-2$.  Begin by removing any $(k-2)$-occurrences of $241\da3$ within $\pi \in \mathcal{A}_{142\da3}(n,k)$ by interchanging the positions of the two smallest letters within these occurrences, starting with the leftmost and working from left to right.  Subsequently remove any $i$-occurrences of $241\da3$ for $i<k-2$ from the word that results in a similar manner, ending with $i=2$.  The mapping so obtained is seen to be invertible.
\end{proof}

Note that the bijections used in the proof of the prior theorem preserve the number of occurrences of each letter and hence show the strong equivalence.

\subsection{Table of equivalence classes for $(3,1)$ and $(2,2)$ patterns}

Combining the previous results yields a
complete solution to the problem of identifying all of the Wilf-equivalence classes
for patterns of type $(3,1)$ and $(2,2)$. By symmetry, this identifies all of the Wilf-equivalence classes for vincular patterns of length four containing a single dash. Below are the lists of the non-trivial $(3,1)$ and $(2,2)$ equivalences.  Others not listed hold due to symmetry.  The Wilf-classes of size one, which may be distinguished from all others by numerical evidence, are also not included.

\begin{itemize}
\item $112\da1\simref{\ref{rem31}}122\da1\simref{\ref{rem11}}122\da2$\\[-4pt]
\item $112\da3\simref{\ref{th8}}122\da3\simref{\ref{th8}}211\da3\simref{\ref{th8}}221\da3$\\[-4pt]
\item $113\da2 \simref{\ref{th3}}133\da2$\\[-4pt]
\item $131\da2 \simref{\ref{th10}}121\da3 \simref{\ref{th8}}212\da3$\\[-4pt]
 \item $123\da1\simref{\mbox{\scriptsize\cite{Ka}}}123\da3\simref{\ref{th6}}123\da2$\\[-4pt]
\item $123\da4\simref{\ref{th8}}321\da4$\\[-4pt]
\item $124\da3\simref{\mbox{\scriptsize\cite{Ka}}}134\da2\simref{\ref{th4}}143\da2\simref{\ref{th4}}214\da3$\\[-4pt]
\item $132\da1\simref{\ref{th1}}132\da2$\\[-4pt]
\item $213\da4\simref{\ref{th8}}231\da4\simref{\ref{th8}}312\da4\simref{\ref{th8}}132\da4\simref{\ref{th10}}142\da3\simref{\ref{th4}}241\da3$\\[-4pt]
\item $213\da1\simref{\ref{rem11}}213\da2.$\\[-4pt]
\end{itemize}

We remark that Theorem \ref{th6} also applies to the equivalence $123\da1\sim123\da3$ found in \cite{Ka}, and that the equivalence $124\da3 \sim 134\da2$ is a special case of Theorem \ref{th11}.

The $(2,2)$ equivalence classes are given as follows.

\begin{itemize}
\item $11\da12\simref{\ref{th9}}11\da21$\\[-4pt]
\item $11\da23\simref{\ref{th8}}11\da32$\\[-4pt]
\item $12\da13 \simref{\mbox{\scriptsize\cite{Ka}}} 13\da12$\\[-4pt]
\item $12\da34\simref{\ref{th8}}12\da43\simref{\ref{th8}}21\da43\simref{\ref{th8}}21\da34$\\[-4pt]
\item $12\da32 \simref{\ref{th9}} 12\da23 \simref{\mbox{\scriptsize\cite{Ka}}} 21\da32\simref{\ref{th9}}21\da23$\\[-4pt]
\item $13\da24\simref{\mbox{\scriptsize\cite{Ka}}}24\da13$\\[-4pt]
\item $14\da23\simref{\mbox{\scriptsize\cite{Ka}}}23\da14.$\\[-4pt]
\end{itemize}

Note that the equivalence $12\da23\sim21\da32$ also follows from taking reverse complements in Theorem \ref{th9}.

\section{Enumerative results}

In this section, we find explicit formulas for and/or recurrences satisfied by the generating functions of the sequences $a_\tau(n,k)$ for various $\tau$ where $k$ is fixed.  Let us define the following generating functions in conjunction with a pattern $\tau$:
\begin{align*}
W_{\tau}(x;k)&=\sum_{n\geq0}a_\tau(n,k)x^n,\\
W_{\tau}(x;k|j_1j_2\cdots j_m)&=\sum_{n\geq0}a_\tau(n,k;j_1j_2\cdots j_m)x^n.
\end{align*}

We will make use of the convention that empty sums take the value one and empty products the value zero.  Our first result provides a way of calculating generating functions for vincular patterns from those for subword patterns.

\begin{theorem}\label{then1}
Let $\tau$ be a subword pattern whose largest letter is $r-1$. Define $\tau'=\tau\da r$. Then for all $k\geq r-1$,
$$W_{\tau'}(x;k)=\frac{1}{(1-(r-1)x)\prod_{j=r-1}^{k-1}\left(1-xW_{\tau}(x;j)\right)}.$$
\end{theorem}
\begin{proof}
Let $k\geq r$. Since each $k$-ary word $\pi$ either contains no $k$'s or can be expressed as $\pi=\pi' k \pi''$, where $\pi'$ is a $(k-1)$-ary word, we obtain
$$W_{\tau'}(x;k)=W_{\tau'}(x;k-1)+xW_{\tau}(x;k-1)W_{\tau'}(x;k),$$
which implies
$$W_{\tau'}(x;k)=\frac{W_{\tau'}(x;k-1)}{1-xW_{\tau}(x;k-1)}$$
for all $k\geq r$. Iterating this last relation, we get
$$W_{\tau'}(x;k)=\frac{1}{(1-(r-1)x)\prod_{j=r-1}^{k-1}\left(1-xW_{\tau}(x;j)\right)},$$
since $W_{\tau'}(x;r-1)=\frac{1}{1-(r-1)x}$.
\end{proof}

\begin{example}
By Theorem \ref{then1} and \cite[Section 3]{BM}, we obtain
\begin{align*}
W_{111\da2}(x;k)&=\frac{1}{(1-x)\prod_{j=0}^{k-2}\left(1-x\frac{1+x+x^2}{1-jx(1+x)}\right)},\quad k\geq 1,\\
W_{112\da3}(x;k)&=\frac{1}{(1-2x)\prod_{j=2}^{k-1}\left(1-\frac{x}{1-\frac{1}{x}+\frac{1}{x}(1-x^2)^j}\right)},\quad k\geq2,\\
W_{212\da3}(x;k)&=\frac{1}{(1-2x)\prod_{j=2}^{k-1}\left(1-\frac{x}{1-x-x\sum_{i=0}^{j}\frac{1}{1-ix^2}}\right)},\quad k\geq2,\\
W_{123\da4}(x;k)&=\frac{1}{(1-3x)\prod_{j=3}^{k-1}\left(1-\frac{x}{1-jx-\sum_{i=3}^j\frac{((-1)^{\lfloor(i-3)/3\rfloor}+(-1)^{\lfloor(i-2)/3\rfloor})(-x)^i}{2}\binom{j}{i}}\right)},\quad k\geq3,\\
W_{213\da4}(x;k)&=\frac{1}{(1-3x)\prod_{j=3}^{k-1}\left(1-\frac{x}{1-x-x\sum_{i=0}^{j-2}\prod_{\ell=0}^i(1-\ell x^2)}\right)},\quad k\geq3.\\
\end{align*}
\end{example}

The next four theorems, taken together with the prior example, provide a complete solution to the problem of determining the generating function $W_\tau(x;k)$ in the case when $\tau$ is a pattern of the form $\tau=11a\da b$.  We apply a modification of the scanning-elements algorithm described in \cite{FM} (see also related work in \cite{NaZ,NoZ}).

\begin{theorem}
For $k \geq 1$,
\begin{align*}
W_{111\da1}(x;k)&=\sum_{j=1}^k\frac{(k-j)!\binom{k}{j}(1+x(1+x)+\delta_{j,1} x^3)x^{3(k-j)}}{\prod_{i=j}^k(1-(i-1)x(1+x))}.
\end{align*}
\end{theorem}
\begin{proof}
From the definitions,
\begin{align*}
W_{111\da1}(x;k)&=1+\sum_{i=1}^k W_{111\da1}(x;k|i),\\
W_{111\da1}(x;k|i)&=xW_{111\da1}(x;k)-xW_{111\da1}(x;k|i)+W_{111\da1}(x;k|ii),\\
W_{111\da1}(x;k|ii)&=x^2W_{111\da1}(x;k)-x^2W_{111\da1}(x;k|i)+x^3W_{111\da1}(x;k-1).
\end{align*}
Solving this system yields the recurrence
\begin{align*}
W_{111\da1}(x;k)&=\frac{1+x(1+x)}{1-(k-1)x(1+x)}+\frac{kx^3}{1-(k-1)x(1+x)}W_{111\da1}(x;k-1), \qquad k \geq 1,
\end{align*}
with $W_{111\da1}(x;0)=1$, which we iterate to complete the proof.
\end{proof}

\begin{theorem}
For $k\geq0$,
$$W_{112\da1}(x;k)=1+\sum_{j=0}^{k-1}\left(\sum_{i=j}^{k-1}\binom{i}{j}(1-x^2)^{i-j}\right)x^{2j+1}W_{112\da1}(x;k-j).$$
\end{theorem}
\begin{proof}
By the definitions,
\begin{align*}
W_{112\da1}(x;k|i)&=xW_{112\da1}(x;k)-xW_{112\da1}(x;k|i)+W_{112\da1}(x;k|ii),\\
W_{112\da1}(x;k|ii)&=x^2+x^2\sum_{j=1}^{i-1}W_{112\da1}(x;k|j)+xW_{112\da1}(x;k|ii)+\sum_{j=i+1}^kW_{112\da1}(x;k|iij)\\
&=x^2+x^2\sum_{j=1}^{i-1}W_{112\da1}(x;k|j)+xW_{112\da1}(x;k|ii)+x^2\sum_{j=i}^{k-1}W_{112\da1}(x;k-1|j),
\end{align*}
which implies
\begin{align*}
W_{112\da1}(x;k|i)=xW_{112\da1}(x;k)-x^2\sum_{j=i+1}^{k}W_{112\da1}(x;k|j)+x^2\sum_{j=i}^{k-1}W_{112\da1}(x;k-1|j)
\end{align*}
for $1 \leq i \leq k$. Therefore, for $0 \leq i \leq k-1$, we have
\begin{align*}
W_{112\da1}(x;k|k-i)-W_{112\da1}(x;k|k-i+1)=-x^2W_{112\da1}(x;k|k-i+1)+x^2W_{112\da1}(x;k-1|k-i).
\end{align*}
Define the array $a_{i,j}$ for $0\leq i \leq j$ by $$W_{112\da1}(x;k|k-i)=\sum_{j=0}^ia_{i,j}W_{112\da1}(x;k-j).$$
The $a_{i,j}$ are seen to be polynomials in $x$. Comparing coefficients of $W_{112\da1}(x;k-j)$ in the last recurrence yields
$$a_{i,j}=(1-x^2)a_{i-1,j}+x^2a_{i-1,j-1}, \qquad i \geq 1 \text{ and } 0 \leq j \leq i,$$
where $a_{0,0}=x$ and $a_{i,j}=0$ if $j>i$ or $j<0$. Define $A_i(y)=\sum_{j=0}^ia_{i,j}y^j$ so that
$$A_i(y)=(1-x^2+x^2y)A_{i-1}(y), \qquad i \geq 1,$$
with $A_0(y)=x$.  Iteration gives
$$A_i(y)=x(1-x^2+x^2y)^i=\sum_{j=0}^i (1-x^2)^{i-j}\binom{i}{j}x^{2j+1}y^j,$$
which implies $a_{i,j}=(1-x^2)^{i-j}\binom{i}{j}x^{2j+1}$. Noting
$$W_{112\da1}(x;k)=1+\sum_{i=0}^{k-1}W_{112\da1}(x;k|k-i)=1+\sum_{j=0}^{k-1}\sum_{i=j}^{k-1}a_{i,j}W_{112\da1}(x;k-j)$$
yields the desired result.
\end{proof}

\begin{theorem}
For $k \geq 1$,
$$W_{112\da2}(x;k)=\sum_{j=1}^k\frac{x}{(1-x^2)^j-1+x}\prod_{i=j+1}^k\frac{ix^2-1+(1-x^2)^i}{(1-x^2)^i-1+x}.$$
\end{theorem}
\begin{proof}
By the definitions,
\begin{align*}
W_{112\da2}(x;k|i)&=xW_{112\da2}(x;k)-xW_{112\da2}(x;k|i)+W_{112\da2}(x;k|ii),\\
W_{112\da2}(x;k|ii)&=x^2+x^2\sum_{j=1}^{i-1}W_{112\da2}(x;k|j)+xW_{112\da2}(x;k|ii)+\sum_{j=i+1}^kW_{112\da2}(x;k|iij)\\
&=x^2+x^2\sum_{j=1}^{i-1}W_{112\da2}(x;k|j)+xW_{112\da2}(x;k|ii)+(k-i)x^3W_{112\da2}(x;k-1),
\end{align*}
for $1 \leq i \leq k$. Thus,
\begin{align*}
W_{112\da2}(x;k|i)&=x^2+x(1-x)W_{112\da2}(x;k)+(k-i)x^3W_{112\da2}(x;k-1)+x^2\sum_{j=1}^iW_{112\da2}(x;k|j),
\end{align*}
which implies
\begin{align*}
W_{112\da2}(x;k|1)&=\frac{x^2}{1-x^2}+\frac{x(1-x)}{1-x^2}W_{112\da2}(x;k)+\frac{(k-1)x^3}{1-x^2}W_{112\da2}(x;k-1),\\
W_{112\da2}(x;k|i)&=\frac{1}{1-x^2}W_{112\da2}(x;k|i-1)-\frac{x^3}{1-x^2}W_{112\da2}(x;k-1), \qquad i>1.
\end{align*}
An induction on $i$ yields
\begin{align*}
W_{112\da2}(x;k|i)&=\frac{x^2}{(1-x^2)^i}+\frac{x(1-x)}{(1-x^2)^i}W_{112\da2}(x;k)+x\left(\frac{kx^2-1}{(1-x^2)^i}+1\right)W_{112\da2}(x;k-1)
\end{align*}
for all $k\geq2$. An induction on $k$ now completes the proof.
\end{proof}

Let $T_i(t)=2tT_{i-1}(t)-T_{i-2}(t)$ if $i\geq 2$ with $T_0(t)=1$ and $T_1(t)=t$ denote the Chebyshev polynomial of the first kind and $U_i(t)$ denote the Chebyshev polynomial of the second kind satisfying the same recurrence but with initial conditions $U_0(t)=1$ and $U_1(t)=2t$ (see, e.g., \cite{R}).  The following result provides a connection between Chebyshev polynomials and the pattern $113\da2$ ($\sim133\da2$).

\begin{theorem}
For $k \geq 0$,
$$W_{113\da2}(x;k)=1+\sum_{j=0}^{k-1}\sum_{i=j}^{k-1}a_{i,j}W_{113\da2}(x;k-j),$$
where $a_{i,j}$ denotes the coefficient of $y^j$ in the polynomial
$$\frac{x(x^2(1-y)+y)^{i/2}}{1+y}\left(2yT_i\left(\frac{1+y}{2\sqrt{x^2(1-y)+y}}\right)+(1-y)U_i\left(\frac{1+y}{2\sqrt{x^2(1-y)+y}}\right)\right).$$
\end{theorem}
\begin{proof}
By the definitions,
$$W_{113\da2}(x;k|i)=xW_{113\da2}(x;k)-xW_{113\da2}(x;k|i)+W_{113\da2}(x;k|ii)$$
and
\begin{align*}
W_{113\da2}(x;k|ii)&=x^2+x^2\sum_{j=1}^{i-1}W_{113\da2}(x;k|j)+xW_{113\da2}(x;k|ii)\\
&\quad+x^2W_{113\da2}(x;k|i+1)+x^2\sum_{j=i+1}^{k-1}W_{113\da2}(x;j|i+1),
\end{align*}
for $1\leq i\leq k$.  Thus,
\begin{align*}
&(1-x^2)W_{113\da2}(x;k|i)-x(1-x)W_{113\da2}(x;k)\\
&\qquad=x^2+x^2\sum_{j=1}^{i-1}W_{113\da2}(x;k|j)+x^2W_{113\da2}(x;k|i+1)+x^2\sum_{j=i+1}^{k-1}W_{113\da2}(x;j|i+1),
\end{align*}
which is equivalent to
\begin{align*}
W_{113\da2}(x;k|i)-xW_{113\da2}(x;k)=-x^2\sum_{j=i+2}^{k}W_{113\da2}(x;k|j)+x^2\sum_{j=i+1}^{k-1}W_{113\da2}(x;j|i+1)
\end{align*}
since $W_{113\da2}(x;k)=1+\sum_{\ell=1}^kW_{113\da2}(x;k|\ell)$.  Replacing $i$ with $i+1$ and subtracting gives
\begin{align*}
W_{113\da2}(x;k|i+1)-W_{113\da2}(x;k|i)&=x^2W_{113\da2}(x;k|i+2)+x^2\sum_{j=i+2}^{k-1}W_{113\da2}(x;j|i+2)\\
&\quad-x^2\sum_{j=i+1}^{k-1}W_{113\da2}(x;j|i+1).
\end{align*}
Replacing $k$ with $k+1$ and subtracting in the last equation gives
\begin{align*}
W_{113\da2}(x;k+1|i+1)-W_{113\da2}(x;k|i+1)&-W_{113\da2}(x;k+1|i)+W_{113\da2}(x;k|i)\\
&=x^2W_{113\da2}(x;k+1|i+2)-x^2W_{113\da2}(x;k|i+1),
\end{align*}
which we rewrite as
\begin{align*}
W_{113\da2}(x;k+1|k-i)&=W_{113\da2}(x;k+1|k-i+1)+W_{113\da2}(x;k|k-i)-x^2W_{113\da2}(x;k+1|k-i+2)\\
&\quad-(1-x^2)W_{113\da2}(x;k|k-i+1),
\end{align*}
where $1 \leq i \leq k-1$.
Define the array $a_{i,j}$ for $0 \leq j \leq i$ by
$$W(x;k|k-i)=\sum_{j=0}^i a_{i,j}W(x;k-j).$$
The $a_{i,j}$ are seen to be polynomials in $x$.  Comparing coefficients of $W_{113\da2}(x;k-j)$ in the last recurrence gives
$$a_{i+1,j}=a_{i,j}+a_{i,j-1}-x^2a_{i-1,j}-(1-x^2)a_{i-1,j-1},\quad 1 \leq i \leq k-2 \text{ and } 0 \leq j \leq i+1,$$
where $a_{0,0}=a_{1,0}=x$, $a_{1,1}=0$ and $a_{i,j}=0$ if $j>i$ or $j<0$.  Define $A_i(y)=\sum_{j=0}^ia_{i,j}y^j$ so that
$$A_{i+1}(y)=(1+y)A_i(y)-(x^2(1-y)+y)A_{i-1}(y), \qquad i \geq 1,$$
with $A_0(y)=A_1(y)=x$. By induction, one can show for $i \geq0$ that
$$A_i(y)=\frac{x(x^2(1-y)+y)^{i/2}}{1+y}\left(2yT_i\left(\frac{1+y}{2\sqrt{x^2(1-y)+y}}\right)+(1-y)U_i\left(\frac{1+y}{2\sqrt{x^2(1-y)+y}}\right)\right),$$
which yields the desired formula.
\end{proof}

Similar techniques also apply to the patterns $121\da1$ and $121\da2$, the results of which we state without proof.

\begin{theorem}
For $k\geq1$,
$$W_{121\da1}(x;k)=\sum_{j=1}^{k}\frac{1}{1-\sum_{\ell=0}^{j-1}\frac{x}{1+\ell x^2}}\prod_{i=j+1}^k\frac{\sum_{\ell=0}^{i-1}\frac{\ell x^3}{1+\ell x^2}}{1-\sum_{\ell=0}^{i-1}\frac{x}{1+\ell x^2}}.$$
\end{theorem}

\begin{theorem}
For $k \geq 0$,
$$W_{121\da2}(x;k)=1+\sum_{j=0}^{k-1}\left(\sum_{i=j}^{k-1}\frac{\binom{i}{j}}{\prod_{\ell=i-j}^i(1+\ell x^2)}\right)j!x^{2j+1}W_{121\da2}(x;k-j).$$
\end{theorem}

Given the quantities $x_1,x_2,\ldots,x_n$ and $1 \leq m \leq n$, let $e_m(x_1,x_2,\ldots,x_n)$ denote the $m$-th elementary symmetric function in those quantities defined by
$$e_m(x_1,x_2,\ldots,x_n)=\sum_{1\leq i_1<\cdots<i_m \leq n}x_{i_1}x_{i_2}\cdots x_{i_m}.$$
Furthermore, we take $e_m(x_1,x_2,\ldots,x_n)$ to be one if $m=0$ and zero if $m>n$.  Our next three results make use of symmetric functions to describe recurrences satisfied by the generating functions.

\begin{theorem}
For $k \geq 0$,
$$W_{132\da3}(x;k)=1+\sum_{j=0}^{k-1}\sum_{i=j}^{k-1}a_{i,j}W_{132\da3}(x;k-j),$$
where
\begin{align*}
a_{i,j}&=x\sum_{d=0}^{i}\binom{i}{d}e_{i-j}((d-1)x^2-1,(d-2)x^2-1,\ldots,-1).
\end{align*}
\end{theorem}
\begin{proof}
By the definitions,
\begin{align*}
W_{132\da3}(x;k|i)&=x+x\sum_{j=1}^{i+1}W_{132\da3}(x;k|j)+\sum_{j=i+2}^kW_{132\da3}(x;k|ij),\\
W_{132\da3}(x;k|ij)&=x^2+x^2\sum_{\ell=1}^iW_{132\da3}(x;k|\ell)+x^2\sum_{\ell=i+1}^{j-1}W_{132\da3}(x;k-1|\ell)\\
&\quad+x^2W_{132\da3}(x;k|j)+x^2W_{132\da3}(x;k|j+1)+x\sum_{\ell=j+2}^kW_{132\da3}(x;k|j\ell)\\
&=xW_{132\da3}(x;k|j)+x^2\sum_{\ell=i+1}^{j-1}W_{132\da3}(x;k-1|\ell)-x^2\sum_{\ell=i+1}^{j-1}W_{132\da3}(x;k|\ell), \qquad j>i+1,
\end{align*}
which implies
\begin{align*}
W_{132\da3}(x;k|i)&=xW_{132\da3}(x;k)+x^2\sum_{\ell=i+1}^{k-1}(k-\ell)W_{132\da3}(x;k-1|\ell)-x^2\sum_{\ell=i+1}^{k}(k-\ell)W_{132\da3}(x;k|\ell).
\end{align*}
Write $W_{132\da3}(x;k|k-i)=\sum_{j\geq0} a_{i,j}W_{132\da3}(x;k-j)$ for some polynomials $a_{i,j}$ in $x$, where $0 \leq j \leq i$. Replacing $i$ by $i-1$ in the last equation and subtracting gives
\begin{align*}
W_{132\da3}(x;k|k-i)-W_{132\da3}(x;k|k-i-1)&=ix^2W_{132\da3}(x;k|k-i)-ix^2W_{132\da3}(x;k-1|k-i),
\end{align*}
upon replacing $i$ with $k-i$.
Comparing coefficients of $W_{132\da3}(x;k-j)$ in the last recurrence gives
\begin{align*}
a_{i+1,j}&=(1-ix^2)a_{i,j}+ix^2a_{i-1,j-1}, \qquad 1 \leq i \leq k-2 \text{ and } 0 \leq j \leq i+1,
\end{align*}
where $a_{0,0}=a_{1,0}=x$, $a_{1,1}=0$, and $a_{i,j}=0$ if $j>i$ or $j<0$.
Define $A_i(y)=\sum_{j\geq0}a_{i,j}y^j$ so that
$$A_{i+1}(y)=(1-ix^2)A_i(y)+ix^2yA_{i-1}(y), \qquad i \geq 1,$$
with $A_0(y)=A_1(y)=x$.

Now define $A(z,y)=\sum_{i\geq0}A_i(y)\frac{z^i}{i!}$. Then the last recurrence can be expressed as
$$(1+zx^2)\frac{d}{dz}A(z,y)=(1+zx^2y)A(z,y),$$
with $A(0,y)=x$. Hence,
\begin{align*}
A(z,y)&=\frac{xe^{zy}}{(1+zx^2)^{(y-1)/x^2}}\\
&=x\sum_{a\geq0}\frac{z^ay^a}{a!}\sum_{b\geq0}\binom{(y-1)/x^2+b-1}{b}z^bx^{2b},
\end{align*}
which implies
\begin{align*}
A_i(y)&=x\sum_{a=0}^i\binom{i}{a}(y-1+(a-1)x^2)(y-1+(a-2)x^2)\cdots(y-1+0\cdot x^2)y^{i-a}\\
&=x\sum_{a=0}^i\sum_{b=0}^a \binom{i}{a}e_b((a-1)x^2-1,(a-2)x^2-1,\ldots,-1)y^{i-b}.
\end{align*}
Extracting the coefficient of $y^j$ from the last expression gives
$$a_{i,j}=x\sum_{a=0}^i\binom{i}{a}e_{i-j}((a-1)x^2-1,(a-2)x^2-1,\ldots,-1),$$
which completes the proof.
\end{proof}

Similar techniques apply to the patterns $132\da1$ ($\sim132\da2$) and $231\da3$, the results of which we state without proof.

\begin{theorem}
For $k \geq 0$,
$$W_{132\da1}(x;k)=1+\sum_{j=0}^{k-1}\sum_{i=j}^{k-1}a_{i,j}W_{132\da1}(x;k-j),$$
where
\begin{align*}
a_{i,j}=e_j\left(\frac{0}{1-0\cdot x^2},\ldots,\frac{i-1}{1-(i-1)x^2}\right)x^{2j+1}\prod_{s=0}^{i-1}(1-sx^2).
\end{align*}
\end{theorem}

\begin{theorem}
For $k \geq 0$,
$$W_{231\da3}(x;k)=1+\sum_{j=0}^{k-1}\sum_{i=j}^{k}a_{i,j}W_{231\da3}(x;k-j),$$
where
\begin{align*}
a_{i,m}&=\frac{k-i-1}{k-1}x^{2m+1}e_m\left(\frac{k-1}{1-(k-1)x^2},\ldots,\frac{k-i}{1-(k-i)x^2}\right)\prod_{s=1}^i(1-(k-s)x^2)\\
&\quad+x^{2m+1}\sum_{j=1}^{i}\frac{(k-i-1)e_m\left(\frac{k-j-1}{1-(k-j-1)x^2},\ldots,\frac{k-i}{1-(k-i)x^2}\right)}{(k-j-1)(k-j)}\prod_{s=j+1}^i(1-(k-s)x^2).
\end{align*}
\end{theorem}

\textbf{Acknowledgement.}  We wish to thank Andrew Baxter for useful discussions and pointing out to us the reference \cite{Ka}.

\end{document}